\pdfoutput=1
\documentclass[reqno]{amsart}
\usepackage[utf8]{inputenc}
\usepackage{amsmath,amsthm,amssymb}
\usepackage{hyperref}

\usepackage[
    backend=biber,
    style=numeric,
    citestyle=nature,
    maxcitenames=9,
    maxbibnames=9,
    sorting=nyt,
    natbib=true,
    url=false, 
    doi=false,
    eprint=true
]{biblatex}
\date{\today}
\usepackage{paralist}

\theoremstyle{definition}
\newtheorem{theorem}{Theorem}[section]

\newtheorem{theorem*}{Theorem}

\newtheorem{corollary}[theorem]{Corollary}

\theoremstyle{definition}
\theoremstyle{definition}
\newtheorem{definition}[theorem]{Definition}
\newtheorem*{definition*}{Definition}

\usepackage{hyperref}
\usepackage[textsize=tiny]{todonotes}
\usepackage{comment}
\begin{document}

\newcommand{\Pos}{Pos}
\renewcommand{\Im}{\operatorname{Im}}

\newcommand{\Hess}{\operatorname{Hess}}
\newcommand{\Id}{\operatorname{Id}}
\newcommand{\Sym}{\operatorname{Sym}}
\newcommand{\vol}{\operatorname{vol}}
\newcommand{\Hom}{\operatorname{Hom}}

\newcommand{\argmin}{\operatorname{argmin}}
\newcommand{\diag}{\operatorname{diag}}

\renewcommand{\Pos}{\operatorname{Pos}}
\newcommand{\calPos}{\mathcal P}
\newcommand{\Int}{\operatorname{Int}}
\setcounter{tocdepth}{1}

\begin{abstract}
We describe the harmonic interpolation of convex bodies, and prove a strong form of the Brunn-Minkowski inequality and characterize its equality case.  As an application we improve a theorem of Berndtsson on the volume of slices of a pseudoconvex domain.  We furthermore apply this to prove subharmonicity of the expected absolute value of the determinant of a matrix of random vectors through the connection with zonoids.
\end{abstract}

%\title[{On the harmonic interpolation of convex sets and zonoids} ]{On the harmonic interpolation of convex sets and a Brunn-Minkowski Theorem for Random Determinants} 

\title[Harmonic Interpolation, Zonoids, Random Determinants]{Harmonic Interpolation and a Brunn-Minkowski Theorem for Random Determinants}

\author{Julius  Ross}
\author{David Witt Nystr\"om}

\address{Mathematics Statistics and Computer Science, University of Illinois at Chicago, Chicago  IL, USA}
\email{juliusro@uic.edu}

\address{Department of Mathematical Sciences, Chalmers University of Technology and the University of Gothenburg, Sweden}
\email{wittnyst@chalmers.se, danspolitik@gmail.com}


\maketitle

\section{Introduction}

Let $A$ and $B$ be convex subsets of $\mathbb{R}^n$. The Minkowski sum of $A$ and $B$ is defined as $$A+B:=\{a+b: a\in A,b\in B\},$$ and the famous Brunn-Minkowski inequality says that $$|A+B|^{1/n}\geq |A|^{1/n}+|B|^{1/n},$$ where $|\cdot|$ denotes the Euclidean volume.

We wish to consider the interpolation of convex sets.   Given convex $A$ and $B$ there is a natural interpolating family $A_t:=(1-t)A+tB$, $t\in[0,1]$, and it follows from the Brunn-Minkowski inequality that the map $$t\mapsto |A_t|^{1/n}$$ is concave in $t\in [0,1]$.

For an infinite family of convex sets, there are many possible  interpolations.  To consider this in more detail, suppose $\Omega$ is a smoothly bounded domain in $\mathbb R^m$ and that we have a continuous family of convex bodies (i.e.\ compact convex sets) $A_{\tau}\subset \mathbb R^n$ parametrized by $\tau\in \partial \Omega$.   If $\Omega$ is itself convex, a natural interpolation can be obtained by considering $$A = \operatorname{Convexhull}\left(\bigcup_{\tau\in \partial \Omega}A_{\tau}\times \{\tau\}\right)\subseteq \mathbb{R}^{n+m}$$ and letting $A_x$ be the fiber of $A$ over $x\in \Omega$.  We call this the \emph{convex interpolation} of $\{A_{\tau}\}$. Then directly from the Brunn-Minkowski inequality it follows that the map $x\mapsto |A_x|^{1/n}$ is concave in $x\in \Omega$.

If $\Omega$ is not convex, the convex interpolation is not suitable since it will not necessarily agree with the given boundary data $\{A_{\tau}\}$ on $\partial \Omega$.  For general $\Omega$ a natural interpolation was proposed in our recent paper \cite{ross2022interpolation} that we now describe.

First note that if $A_y$ is a continuous family of convex bodies in $\mathbb R^n$ over some parameter set $D\subseteq \mathbb{R}^m$ and $\mu$ is a Radon measure on $D$, then there is a set-integral $$\int_D A_yd\mu(y)$$ 
which is itself a subset of $\mathbb R^n$.   To define this precisely recall that the support function of a convex set $A$ is given by

$$h_A(\xi):=\sup_{\zeta\in A}(\zeta \cdot \xi)$$ with and has the property that
\begin{equation} h_A \text{ is convex and } h_A(t\xi)=|t|h_A(\xi) \text{ for } t\in \mathbb R.\label{eq:supportproperties}\end{equation}
On the other hand, if $h$ is a function with those two properties then $h$ is the support function of a unique closed convex set, which we denote by $A(h)$.

It is an elementary exercise to see that $h_{A+B}=h_A+h_B$ and more generally $$h_{t_1A_1+...+t_kA_k}=t_1h_{A_1}+...+t_kh_{A_k}.$$
Furthermore if $A_t$ are convex sets such that $A_t\to A$ in the Hausdorff topology, then for each $\xi$, $h_{A_t}(\xi)\to h_A(\xi)$.    

\begin{definition}
Let $d\mu$ be a measure on a measurable set $D$ in $\mathbb R^m$, and $A_y$ be a convex set for each $y\in D$.    We define the \emph{Minkowski integral} $\int_D A_yd\mu(y)$ as $$\int_D A_yd\mu(y):=A\left(\int_D h_{A_y}d\mu(y)\right).$$
\end{definition}

Such set-valued integrals have been considered in various places, for example \cite{Aumann65,Vitale91}.  As one would expect, some conditions are needed to ensure that the Minkowski integral is well-defined.  For our purpose the following is sufficient: assume $D$ is compact, $d\mu$ is a Radon measure and $y\mapsto A_y$ is a continuous family of convex bodies.  Then for each $\zeta$ the map $y\mapsto h_{A_y}(\zeta)$ is continuous, so $\int_D h_{A_y}d\mu(y)$ exists and has properties \eqref{eq:supportproperties}, and thus $\int_D A_y d\mu(y)$ exists.

%Using this we can describe the harmonic interpolation of a family $A_{\tau}$ over $\partial \Omega$.

\begin{definition}
Let $\Omega\subset \mathbb R^m$ be a smoothly bounded domain. The \emph{harmonic interpolation} of a continuous family $\{A_{\tau}\}_{\tau\in \partial \Omega}$ of convex bodies is defined as $$A_x:=\int_{\partial \Omega}A_{\tau}d\mu_x(\tau),$$ where $d\mu_x$ is the harmonic measure on $\partial \Omega$ with respect to $x\in \Omega$.
\end{definition}

The harmonic interpolation and convex interpolation may differ, even when $\Omega$ is convex.   We argue that the former is better suited in some contexts, one of which is the theory of zonoids.

A \emph{zonotope} is a convex set that can be written as the Minkowski sum of line segments. Clearly any zonotope is a convex polytope, but it is easy to see that not all convex polytopes are zonotopes. A \emph{zonoid} is a convex set which can be approximated arbitrarily well (in the Hausdorff topology) by zonotopes, or equivalently a convex set that can be written as the Minkowski integral of line segments (see for example \cite{Schneider-Weil-Zonoidsandrelatedtopics} for a introduction to zonoids). The harmonic interpolation has the property that it preserves zonoids; i.e. if each boundary set $A_{\tau}$ is a zonoid then each member of the interpolating family $A_x$ will also be a zonoid (and this is not true for the convex interpolation, even when $\Omega$ is convex).

\subsection*{Acknowledgements} This material is based upon work supported by the National Science Foundation under Grant No. DMS-1749447. The second named author is supported by the Swedish Research Council and the G\"oran Gustafsson Foundation for Research in Natural Sciences and Medicine. The authors thank Bo Berndtsson and Dario Cordero-Erausquin for conversations on this topic.

\section{A Brunn-Minkowski Inequality for Harmonic Interpolation}

We continue to assume $\Omega\subset \mathbb R^m$ is a smoothly bounded domain (which by convention is also bounded), and $A_{\tau}$ for $\tau\in \partial \Omega$ is a continuous family of convex bodies in $\mathbb R^n$.  In \cite{ross2022interpolation} we proved the following weak version of a Brunn-Minkowski inequality for the harmonic interpolation.

\begin{theorem}
If $A_x$ is the harmonic interpolation of $\{A_{\tau}\}$ then $x\mapsto \log|A_x|$ is superharmonic in $x$.
\end{theorem}

Our main result in this short note is a direct proof of the following stronger version:

\begin{theorem} \label{thm:main}
If $A_x$ is the harmonic interpolation of $\{A_{\tau}\}$ then $x\mapsto |A_x|^{1/n}$ is superharmonic in $x$.
\end{theorem}

\begin{proof}
Let $B_{\epsilon}(x)$ denote the Euclidean ball of radius $\epsilon$ centered at $x$.  We need to show that if $B_{\epsilon}(x)\subseteq \Omega$ then $$|A_x|^{1/n}\geq \int_{\partial B_{\epsilon}(x)}|A_y|^{1/n}dS(y),$$ where $dS$ denotes the normalized Euclidean surface measure on $\partial B_{\epsilon}(x)$.

A standard property of harmonic measures is that $$\mu_x=\int_{\partial B_{\epsilon}(x)}\mu_ydS(y),$$ and this clearly implies that $$A_x=\int_{\partial B_{\epsilon}(x)}A_ydS(y).$$   Now we approximate the surface measure $dS$ with a sequence of atomic measure $\nu_k=\sum_{i=1}^{N_k} \lambda_{i,k} \delta_{y_{i,k}}$ chosen so $\nu_k\to dS$ weakly as $k\to \infty$.  Then for $k$ sufficiently large $\sum_{i=1}^{N_k} \lambda_{i,k} A_{y_i}$ is arbitrarily close (in the Hausdorff distance) to $A_x$.  Thus for any $\epsilon>0$ and $k$ sufficiently large we have
\begin{eqnarray*}
|A_x|^{1/n}+ \epsilon \ge |\sum_{i=1}^{N_k} \lambda_{i,k} A_{y_{i,k}}|^{1/n}\geq \sum_{i=1}^{N_k} \lambda_{i,k}|A_{y_{i,k}}|^{1/n}\ge \int_{\partial B_{\epsilon}(x)}|A_y|^{1/n}dS(y) - \epsilon,
\end{eqnarray*}
where the second follows from the classical Brunn-Minkowski inequality. Letting $\epsilon\to 0$ we have $$|A_x|^{1/n}\geq \int_{\partial B_{\epsilon}(x)}|A_y|^{1/n}dS(y),$$ which completes the proof.
\end{proof}

%A look at the proof shows more, in that we only used the following property of the harmonic interpolation:

\begin{definition}
 We say that a continuous family of convex sets $A_x\subseteq \mathbb{R}^n$ over some domain $\Omega\subseteq \mathbb{R}^m$ is \emph{subharmonic} over $\Omega$ if whenever $B_{\epsilon}(x)\subseteq \Omega$ we have that $$A_x\supseteq\int_{\partial B_{\epsilon}(x)}A_ydS(y).$$
\end{definition} 

The then get the following Corollary of Theorem \ref{thm:main}.

\begin{corollary} \label{cor:sub}
If $A_x$ is subharmonic over $\Omega$ then $x\mapsto |A_x|^{1/n}$ is superharmonic.
\end{corollary}

As an application we can give a strengthening of the following theorem of Berndtsson \cite{Berndtsson_Prekopa}

\begin{theorem}\label{thm:B}
Let $U\subseteq \mathbb{C}^{n+m}$ be a pseudoconvex domain with the property that if $(x_1+iy_1,...,x_n+iy_n,w)\in U$ then $(x_1+iy'_1,...,x_n+iy'_n,w)\in U$ for all $y_1',\ldots,y_n'$, and let $U_w:=\{x\in \mathbb{R}^n: (x,w)\in U\}$. 

Then the map $w\mapsto -\log|U_w|$ is plurisubharmonic in $w$.
\end{theorem}

%We now show the following strengthening:

\begin{theorem}
In the same setting as Theorem \ref{thm:B}, the map $w\mapsto -|U_w|^{1/n}$ is plurisubharmonic.
\end{theorem}

\begin{proof}
    Without loss of generality we can assume that $m=1$. Note that the pseudoconvexity and symmetry of $U$ implies that $U_w$ is convex for all $w$. By approximation we can also without loss of generality assume that the family $U_w$ is bounded and continuous. We claim that the family $U_w$ is subharmonic. Note that for two closed convex sets $A$ and $B$ we have that $A\supseteq B$ if and only if $h_A\geq h_B$, so $U_w$ is subharmonic if and only if for any $\xi\in \mathbb{R}^n$, $h_{U_w}(\xi)=\sup_{x\in U_w}(x\cdot \xi)$ is superharmonic in $w$. 
    
    Let $\phi$ be a plurisubharmonic exhaustion function for $U$ which we can assume to be independent of $\Im(\mathbb{C}^n)$, just as $U$ itself. Note that $\phi_R:=\max(\phi-R,0)$  is also plurisubharmonic and independent of $\Im(\mathbb{C}^n)$, and that the same is true for $\psi_R(x+iy,w):=\phi_R(x+iy,w)-x\cdot \xi$. Thus by Kiselman's minimum principle $\inf_{x\in U_w}\psi_R(x,w)$ is subharmonic in $w$. We now note that $$\sup_{x\in U_w}(x\cdot \xi)=-\lim_{R\to\infty}\inf_{x\in U_w}\psi_R(x,w),$$ and hence it follows that $h_{U_w}(\xi)$ is superharmonic and thus $U_w$ is subharmonic. That $-|U_w|^{1/n}$ is subharmonic now follows from Corollary \ref{cor:sub}. 
\end{proof}

\section{Characterization of the extremal case}

%A characterization of subharmonic families such that $|A_x|^{1/n}$ is harmonic.
%
%Assume that $A_x$ is a subharmonic family of convex sets in $\mathbb{R}^n$ over some domain $\Omega\subseteq \mathbb{R}^m$. 

By Corollary \ref{cor:sub} we know that if $\{A_x\}_{x\in \Omega}$ is a subharmonic family of convex bodies in $\mathbb R^n$ over a domain $\Omega$ then $x\mapsto |A_x|^{1/n}$ is superharmonic.   Our next result characterizes when this map is in fact harmonic.

\begin{theorem}
The map $x\mapsto |A_x|^{1/n}$ is harmonic if and only if we can write $A_x=c_xB+d_x$ where $B\subset \mathbb R^n$ is a fixed convex body, and $c_x$ and $d_x$ are harmonic functions on $\Omega$ taking values in $\mathbb{R}_+$ and $\mathbb{R}^n$ respectively. 
\end{theorem}

\begin{proof}
Let $\Omega'$ be a relatively compact subdomain of $\Omega$ with smooth boundary. Since $A_x$ is subharmonic it must dominate the harmonic interpolation of $A_y$ restricted to $\partial \Omega'$, but since $|A_x|^{1/n}$ is assumed to be harmonic we must have that $A_x$ is equal to the harmonic interpolation.

Write $\partial \Omega'$ as the disjoint union of a finite number of measurable subsets $D_i$ and let $$B_i:=\int_{D_i}A_yd\mu_x(y),$$ where $\mu_x$ is the harmonic measure on $\partial \Omega'$ with respect to $x$. Then  $A_x=\sum_iB_i$, and by the Brunn-Minkowski inequality we have $$|A_x|^{1/n}\geq \sum_i|B_i|^{1/n}.$$ On the other hand, as in the proof of Theorem \ref{thm:main} one sees that $$|B_i|^{1/n}\geq \int_{D_i}|A_y|^{1/n}d\mu_x(y).$$ But $|A_x|^{1/n}$ being harmonic then implies the equality $$|A_x|^{1/n}= \sum_i|B_i|^{1/n}.$$ The well-known characterization of the equality case of the Brunn-Minkowski inequality then says that we can write $B_i=c_iB+d_i$, where $B$ is some fixed convex set, and some $c_i\in \mathbb{R}_+$ and $d_i\in \mathbb{R}^n$. We may normalize $B$ to have volume one and center of gravity at the origin. We thus also see that $A_x=c_xB+d_x$ where $c_x=\sum_ic_i$ and $d_x=\sum_i d_i$. 

Now if we decompose a fixed $D_i$ further into disjoint pieces $E_j$ the same argument yields that for each $j$ there are $c'_j\in \mathbb{R}_+$ and $d'_j\in \mathbb{R}^n$ such that $\int_{E_j}A_ydS(y)=c'_jB+d'_j$. As we can make the decomposition arbitrarily fine the continuity of $A_y$ implies that there are continuous functions $c_y$ and $d_y$ on $\partial \Omega'$ such that $A_y=c_yB+d_y$. It follows that $A_x=c_xB+d_x$ where $c_x$ is the harmonic extension of $c_y$ and $d_x$ is the harmonic extension of $d_y$ to $\Omega'$.

As this can be done for any for relatively compact subdomain of $\Omega$ with smooth boundary, the result follows. 
\end{proof}

\section{A Brunn-Minkowski theorem for expected absolute random determinants}

Consider now a random $(n,n)$ matrix $M_Y$ whose columns are iid copies of a random vector $Y$, corresponding to a Borel probability measure $\nu_v$ on $\mathbb{R}^n$. We are then interested in the expected absolute value of the determinant (ead) $E|\det M_Y|$ of $M_Y$.

Suppose $Y_{\tau}$ is a family of random vectors parametrized by the boundary of a smoothly bounded domain $\Omega\subseteq \mathbb{R}^m$.  We assume that each $Y_{\tau}$ has finite expectation.  Then a natural interpolating family $Y_x$ over $\Omega$ is given by letting $$\nu_{Y_x}:=\int_{\partial \Omega}\nu_{Y_{\tau}}d\mu_x(\tau),$$ where as before $d\mu_x$ denotes the harmonic measure with respect to $x$.

\begin{theorem} \label{thm:random}
The map $x\mapsto (E|\det M_{Y_x}|)^{1/n}$ is superharmonic in $x$.   
\end{theorem}

Our proof relies on the connection between eads and a special class of convex sets called zonoids which was established in \cite[Thm 3.1]{Vitale91}: to any random vector $Y$ with finite expectation we may associate a zonoid  $$Z(Y):=\int_{\mathbb{R}^n}[0,y]d\nu_{v}(y).$$  Then the main result \cite[Thm.\ 3.2]{Vitale91} says that 

\begin{equation} \label{eq:zonvol}
 E|\det M_Y|=n!|Z(Y)|.   
\end{equation} 

%We are now ready to give the proof of Theorem \ref{thm:random}.

\begin{proof}[Proof of Theorem \ref{thm:random}]
Note that 
\begin{eqnarray*}
  Z(Y_x)=\int_{\mathbb{R}^n}[0,y]d\nu_{Y_x}(y)=\int_{\mathbb{R}^n}\int_{\partial{\Omega}}[0,y]d\mu_x(\tau)d\nu_{Y_{\tau}}(y)=\\=\int_{\partial{\Omega}}\int_{\mathbb{R}^n}[0,y]d\nu_{Y_{\tau}}(y)d\mu_x(\tau)=\int_{\partial{\Omega}}Z(Y_{\tau})d\mu_x(\tau),
\end{eqnarray*}
 i.e. $Z(Y_x)$ is the harmonic interpolation of $Z(Y_{\tau})$. Thanks to the volume equality (\ref{eq:zonvol}) the result  follows immediately from Theorem \ref{thm:main}. 
\end{proof}

\printbibliography
\end{document}